\documentclass{amsart}

\usepackage{amsmath}%
\usepackage{amsfonts}%
\usepackage{amssymb}%
\usepackage{graphicx}
\usepackage[latin1]{inputenc}
\usepackage[T1]{fontenc}
\newtheorem{theorem}{Theorem}
\theoremstyle{plain}

\newtheorem{definition}{Definition}
\newtheorem{example}{Example}

\newtheorem{lemma}{Lemma}

\numberwithin{equation}{section}
\begin{document}
\title[Surfaces with quadratic support function of harmonic type]{Surfaces with quadratic support function of harmonic type}

\author{Armando M. V. Corro}
\curraddr[Armando M. V. Corro]{Instituto de Matem\'atica e Estat\'istica, Universidade Federal de Goi\^as, 74001-970, Goi\^ania-GO, Brazil.}%
\email[]{avcorro@gmail.com}%
\author{Carlos M. C. Riveros}
\curraddr[Carlos M. C. Riveros]{Departamento de Matem\'atica, Universidade de Bras\'ilia, CEP 70910-900, Bras\'ilia, DF, Brasil.}%
\email[]{carlos@mat.unb.br}%
\author{Jose L. Teruel Carretero}
\curraddr[Jose L. Teruel Carretero]{Departamento de Matem\'aticas, Facultad de Ciencias, Campus de San Vicente, Universidad de Alicante, CP 03690, San Vicente del Raspeig, Alicante, España.}%
\email[]{jose.teruel@ua.es; jlteruel@mat.unb.br}%

\begin{abstract}
In this paper, we study oriented surfaces $S$ in $\mathbb{R}^3$, called Surfaces with quadratic support function of harmonic type (in short HQSF-surfaces), these surfaces generalize the QSF-surfaces. We obtain a Weierstrass type representation for the HQSF-surfaces which depends on three holomorphic functions. Moreover, classify the HQSF-surfaces of rotation.
\end{abstract}
\maketitle
\section{Introduction}
Let $S \subset \mathbb{R}^{3}$ be a surface oriented by its normal Gauss map $N$. The functions $\Psi, \Lambda : S \rightarrow \mathbb{R}^{3}$ given by $\Psi(p) = \langle p, N(p)\rangle,\,\Lambda(p) = \langle p,p\rangle$, $p \in S$, where $\langle ,\rangle$ denotes the Euclidean scalar product in  $\mathbb{R}^{3}$, are called {\em support function} and {\em quadratic distance function}, respectively.

Appell in \cite{1}, studied a class of oriented surfaces in $\mathbb{R}^3$ associated with area preserving transformations in the sphere.
In \cite{6}, the authors showed that these surfaces are such that the mean curvature $H$, the Gaussian curvature $K$ and the support function $\Psi$ satisfy  $H + \Psi K = 0.$ Tzitz\'eica in \cite{11} estudied oriented hyperbolic surfaces such that there exist a nonzero constant $c \in \mathbb{R}$ for which the following relation is satisfied $K + c^{2}\Psi^{4}= 0$.

A surface $S\subset \mathbb{R}^3$ is a {\em Laguerre minimal surface} if
\[
\bigtriangleup_{III}\left(\frac{H}{K}\right)=0,
\]
where $H$ is the mean curvature, $K$ is the Gaussian curvature  and $III$ is the third fundamental form of $S$. Laguerre minimal surfaces have been studied by several authors, for example see \cite{AGL}, \cite{MN1}, \cite{MN2}, \cite{15}, \cite{YS}, \cite{21} and \cite{22}.

In \cite{CR}, the authors classify the Laguerre minimal surfaces with planar curvature lines. In addition, they define the spherical type hypersurfaces, as one of the generalizations of the Laguerre minimal surfaces for larger dimensions.

In \cite{4}, the authors motivated by the works \cite{1} and \cite{11} they introduce the class of oriented surfaces $S$ in $\mathbb{R}^3$
satisfying a relation of the form $A(p) + B(p)H(p) + C(p)K(p) = 0, p \in S,$ where $A, B, C : S \rightarrow \mathbb{R}$  are differentiable
functions depending on the support function and the quadratic distance function. These surfaces are called surfaces of generalized Weingarten surfaces depending on the support and distance functions (in short, DSGW-surfaces). The linear Weingarten surfaces, Appell's surfaces and Tzitzeica's surfaces are all DSGW-surfaces. In this context, the authors study a special class of oriented surfaces $ S\subset \mathbb{R}^{3} $ that satisfy a relation of the form $ 2 \Psi H + \Lambda K = 0 $, this surfaces are called {\em EDSW-surfaces}. They show that these surfaces are invariant by dilations and inversions. Moreover, they obtain a Weierstrass type representation depending on two holomorphic functions.

A generalization of the DSGW-surfaces is to consider the functions $A, B, C : S \rightarrow \mathbb{R}$ depending on the support function, the quadratic distance function and harmonic functions  with respect to the third quadratic form, these surfaces are called of generalized Weingarten surfaces depending on the support and distance functions of harmonic type (in short DSGWH-surfaces). The Laguerre minimal surfaces is a class of the DSGWH-surfaces, since $H$ and $K$ satify
\[
H-\mu K=0,
\]
where $\mu$ is a harmonic function with respect to the third quadratic form.

In \cite{9}, the authors introduce the class of surfaces in Euclidean space motivated by a problem posed by \'{E}lie Cartan. This class of
surfaces are called {\em Ribaucour surfaces} and are defined as surfaces where all the medial spheres intercept a fixed sphere along a large circle, these surfaces satisfy a relation of the form
\[
2\Psi H + (1+ \Lambda) K= 0.
\]
The authors obtain holomorphic data for these surfaces and discuss the relation with minimal surfaces.

In \cite{CRF}, the authors study  the Ribaucour surfaces of harmonic type (in short HR-surfaces), these surfaces satisfy
\[
2 \Psi H + (ce^{2 \mu}+\Lambda ) K = 0,
 \]
where $ c $ is a nonzero real constant, $ \mu $ a harmonic function with respect to the third fundamental form. These surfaces generalize the Ribaucour surfaces studied in \cite{9}.

Corro and Mendez in \cite{CM}, study the Ribaucour-type surfaces (in short RT-surfaces) what satisfy
\[
2\Psi H+(\Lambda+\Psi^2)K=0.
\]
In \cite{CRC}, the authors study the surfaces with quadratic support function (in short QSF-surfaces), these surfaces satisfy
\[
2\Psi H+(\Lambda-\Psi^2)K=0.
\]
They obtain a Weierstrass type representation for the QSF-surfaces which depends on two holomorphic functions.\\

In this paper, we study the surfaces with quadratic support function of harmonic type (in short HQSF-surfaces), these surfaces satisfy
\[
2\Psi H+(c\Psi e^{2\mu}+\Lambda-\Psi^2)K=0,
\]
where $\mu$ is a harmonic function with respect to the third fundamental form.\\
This class of surfaces generalize the QSF-surfaces studied in \cite{CRC}. We obtain a Weierstrass type representation for the HQSF-surfaces which depends on three  holomorphic functions. Moreover, classify the HQSF-surfaces of rotation.

\section{Preliminaries}

In this section we present the definitions and results that will be used in the work. In this paper the inner
produt $\langle ,\rangle:\mathbb{C}\times\mathbb{C} \rightarrow \mathbb{R}$ is defined by
\[
\langle f,g\rangle= f_1g_1+f_2g_2, \,\,\mbox{where} \,\,f=f_1 +if_2, \,\,g=g_1 +ig_2,
\]
are holomorphic functions.\\
In the computation we use the following properties: if $f, g:\mathbb{C} \rightarrow \mathbb{C}$
are holomorphic functions of $z=u_1+iu_2$, then
\[
\langle f,g\rangle_{,1}=\langle f',g\rangle+\langle f,g'\rangle,\quad  \langle f,g\rangle_{,2}=\langle if',g\rangle+\langle f,ig'\rangle,
\]
\begin{equation}
\label{eq1}
\langle f,g\rangle=\langle 1,\bar{f}g\rangle, \quad \langle 1,f\rangle^2=\frac{1}{2}\left[|f|^2+\langle 1,f^2\rangle\right]
\end{equation}
Here $\langle f, g\rangle_{,i}$ denotes the derivative of $\langle f, g\rangle$ with respect to $u_i,\; i = 1, 2$.

The following lemma was obtained in \cite{CRF}.

\begin{lemma}{If $f_1, f_2, g_1, g_2:\mathbb{C} \rightarrow \mathbb{C}$ are holomorphic functions of $z=u_1+iu_2$, such that
$\,\langle f_1,g_1\rangle+\langle f_2,g_2\rangle=0,$ then
\begin{equation}
\label{eq22}
\,\,g_1= ic_1f_1-\overline{z}_1 f_2, \,\,g_2=z_1 f_1 +ic_2 f_2,
\end{equation}
where $c_i$ are real constants and $z_1 \in \mathbb{C}.$}
\end{lemma}
The next result was obtained in \cite{4}.

\begin{theorem}{ Let $S$ be a surface with Gaussian curvature $K \neq  0$ and $N$ its normal Gauss
map locally parametrized by
\begin{equation}
\label{eq2}
N(u)=\frac{1}{1+|g|^2}(2 g,1-|g|^2)
\end{equation}
where $\;\Pi = \{(u_1,u_2, u_3)\in \mathbb{R}^3: u_3 = 0\}$, $g:U\rightarrow \Pi$ is a holomorphic function with $|g'| \neq 0$, $U$ is a connected open subset of $\mathbb{R}^2$ and $u=(u_1,u_2) \in U$.
Then there is a differentiable function $h:U \rightarrow \mathbb{R}$, such that $S$ can be locally parametrized by
\begin{equation}
\label{eq3}
\displaystyle X(u)= \left(\frac{g'}{|g'|^2}\nabla h-\frac{2R}{T}g,-\frac{2R}{T}\right),
\end{equation}
where
\begin{equation}
\label{eq4}
\displaystyle  T = 1 + |g|^{2}\;\;\mbox{and}\;\;\displaystyle R = \left < \nabla h,\frac{g}{g'}\right > - h.
\end{equation}
Moreover, the coefficients of the first and the second fundamental form of $X$ are given by
\begin{equation}
\label{eq5}
a_{11}=\frac{|g'|^2}{T^2} [ A_1^2 +  (TV_{12})^2], \;\; a_{12}= -\frac{ |g'|^2}{T}V_{12}[A_1+ A_2], \;\;
a_{22}=\frac{|g'|^2}{T^2} [ A_2^2 +  (TV_{12})^2],
\end{equation}
\begin{equation}
\label{eq6}
b_{11}=\frac{2|g'|^2}{T^2}A_1, \;\; b_{12}= -\frac{2|g'|^2}{T}V_{12},\;\; b_{22}=  \frac{2|g'|^2}{T^2}A_2,\;\;A_i = 2R-TV_{ii},\; i=1,2,
\end{equation}
where
\begin{eqnarray}
\nonumber
V_{11}&=&\frac{1}{|g'|^2}\left[h_{,11}-\left <\frac{g''}{g'},\nabla h\right >\right],\\
\label{eq7}
V_{12}&=&\frac{1}{|g'|^2}\left[h_{,12}-\left <i\frac{g''}{g'}, \nabla h\right >\right],\\
\nonumber
V_{22}&=&\frac{1}{|g'|^2}\left[h_{,22}+\left <\frac{g''}{g'},\nabla h\right >\right].
\end{eqnarray}
The regularity condition of $X$ is given by
\begin{equation}
\label{eq8}
P=(A_1A_2-T^2V_{12}^2)\neq 0.
\end{equation}
The third fundamental form is determined by
\begin{equation}
\label{eq9}
L_{ii}=\langle N_{,i}, N_{,i}\rangle=\frac{4}{T^2}|g'|^2,\;i=1,2,\;\; \langle N_{,1}, N_{,2}\rangle=0.
\end{equation}
\begin{equation}
\label{eq10}
H= -\frac{1}{P}\left(T\frac{\Delta h}{|g'|^2}-4R\right), \quad K= \frac{4}{P}
\end{equation}
Conversely, let be a holomorphic function $g:U\rightarrow \Pi$, $U$, where $U$ is a connected open subset of $\mathbb{R}^{2}$ and a
differentiable function $h:U \rightarrow \mathbb{R}$. Then
(\ref{eq3}) define an immersion in $\mathbb{R}^3$ with Gaussian curvature non-zero, Gauss map given by
(\ref{eq2}) and (\ref{eq5})-(\ref{eq10}) are satisfied.}
\end{theorem}

\section{ Surfaces with quadratic support function of harmonic type}
\begin{definition}{\em We say that an oriented surface $S \subset \mathbb{R}^{3}$ is a {\em Surface with quadratic support function of harmonic type (in short, HQSF-surface)} if the mean curvature $H$, the Gaussian curvature $K$, the support function $\Psi$ and quadratic distance $\Lambda$ satisfy
\begin{equation}
\label{eq13}
2\Psi H+(c\Psi e^{2\mu}+\Lambda-\Psi^2)K=0.
\end{equation}}
\end{definition}
The following Theorem provides a Weierstrass type representation for HQSF-surfaces which depends on three  holomorphic functions.
\begin{theorem}{Let $S \subset \mathbb{R}^3$ be a connected orientable Riemann surface. Then $S$ is a HQSF-surface if,
and only if, there exist holomorphic functions $g, f_1, f_2: S \rightarrow \mathbb{C}_{\infty}$, with $|f_1f'_2-f_2f'_1|\neq 0$ such that $X(S)$ is locally parametrized by
\begin{eqnarray}
\nonumber
X(z)&=& \left(\frac{2(|f_1|^2+c|f_2|^2)}{T^2}\left(2T\left(f_1\overline{\left(\frac{f_1'}{g'}\right)}+cf_2\overline{\left(\frac{f_2'}{g'}\right)}\right)-g\left(|f_1|^2+c|f_2|^2\right)\right),0 \right)\\
\label{eq14}
&&-\frac{2R}{T}(g,1),
\end{eqnarray}
where
\begin{equation}
\label{eq15}
R=\frac{(|f_1|^2+c|f_2|^2)}{T^2}\left( 4T\left\langle 1,\frac{(\overline{f_1}f_1'+c\overline{f_2}f_2')g}{g'}\right\rangle-(3|g|^2+1)(|f_1|^2+c|f_2|^2)\right),
\end{equation}
$c \in \mathbb{R}, |g'|\neq 0,\,T=1+|g|^2$.}
\end{theorem}
\begin{proof} 
From  Theorem 1, we have that
\begin{equation}
\label{eq16}
\frac{H}{K}=-\frac{1}{4}\left(\frac{T\Delta h}{|g'|^2}-4R\right),
\end{equation}
\begin{equation}
\label{eq17}
  \Psi=\frac{2h}{T},\,\,\Lambda=\frac{|\nabla h|^2}{|g'|^2}-4R\frac{h}{T},\,L_{11}=\frac{4|g'|^2}{T^2},\,T=1+|g|^2.
\end{equation}
Using (\ref{eq17}) in (\ref{eq16}) we obtain
\begin{equation}
\label{eq18}
\frac{H}{K}=-\frac{1}{4}\left(\frac{T\Delta h}{|g'|^2}+\frac{T}{h}\left(\Lambda-\frac{|\nabla h|^2}{|g'|^2}\right)\right)=-\frac{1}{2\Psi}\left(\frac{h\Delta h-|\nabla h|^2}{|g'|^2}+\Lambda\right).
\end{equation}
Now, let $h=\displaystyle\frac{A^2}{T}$ where $A$ is a differentiable function.\\
We can show that
\begin{equation}
\label{eq20}
h\Delta h-|\nabla h|^2=\frac{2A^2}{T^2}\left(A\Delta A-|\nabla A|^2\right)-\frac{A^4}{T^4}\left(T\Delta T-|\nabla T|^2\right).
\end{equation}
Considering (\ref{eq17}) we obtain
\begin{equation}
\label{eq21}
  T\Delta T-|\nabla T|^2=4|g'|^2,\,\, \Psi=\frac{2A^2}{T^2}.
\end{equation}
Using (\ref{eq20}) and (\ref{eq21}) in (\ref{eq18}) 
\begin{equation}
\label{eq22}
  \frac{H}{K}=-\frac{1}{2\Psi}\left(\Psi\left(\frac{A\Delta A-|\nabla A|^2}{|g'|^2}\right)+\Lambda-\Psi^2\right).
\end{equation}
From (\ref{eq13}) and (\ref{eq22}) $X$ is a HQSF-surface if and only if  
\begin{equation}
\label{eq23}
\frac{A\Delta A-|\nabla A|^2}{|g'|^2}=ce^{2\mu}.
\end{equation}
Now, considering $A=|f_1|^2+c|f_2|^2$, where $f_1$ and $f_2$ are holomorphic functions, we obtain  
\begin{equation}
\label{eq24}
A\Delta A-|\nabla A|^2=4c|f_1f'_2-f_2f'_1|^2.
\end{equation}
From (\ref{eq23})  and (\ref{eq24}) it follows
\begin{equation}
\label{eq19}
\mu=\ln \left(\frac{2|f_1f'_2-f_2f'_1|}{|g'|}\right).
\end{equation}
Therefore, the function $h$ is given by
\begin{equation}
\label{eq25}
h=\frac{(|f_1|^2+c|f_2|^2)^2}{1+|g|^2}.
\end{equation}
Using (\ref{eq25}) we obtain
\begin{equation}
\label{eq26}
\nabla h=2h\left(\frac{2(f_1\overline{f_1'}+cf_2\overline{f_2'})}{|f_1|^2+c|f_2|^2}-\frac{g\overline{g'}}{1+|g|^2}\right).
\end{equation}
Using (\ref{eq26}) in (\ref{eq3}) we obtain  (\ref{eq14}). The proof is complete.
\end{proof}

\begin{example} {\em Considering $f_1(z)=z, f_2(z)=e^z, g(z)=z^2$ and $c=-1$ in Theorem 2, we obtain the HQSF-surface
\begin{figure}[h!]
 \centering
       \begin{minipage}[b]{8.5cm}
        \includegraphics[width=8.5cm]{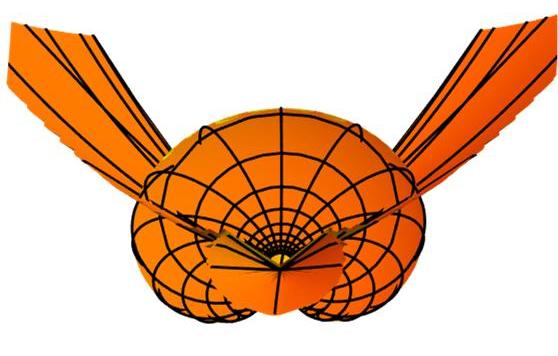}
       \caption{HQSF-surface}
      \end{minipage}
      \end{figure}}
 \end{example}
\newpage
\begin{example} {\em Considering $f_1(z)=\cosh z, f_2(z)=\sinh z, g(z)=z^3$ and $c=2$ in Theorem 2, we obtain the HQSF-surface

\begin{figure}[h!]
 \centering
       \begin{minipage}[b]{10.5cm}
        \includegraphics[width=10.5cm]{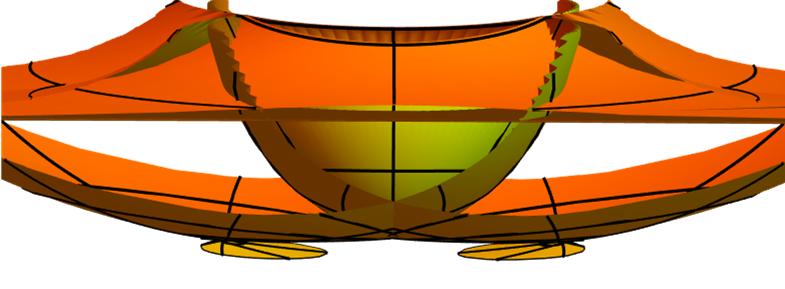}
       \caption{HQSF-surface}
      \end{minipage}
      \end{figure}}
 \end{example}

\begin{example} {\em Considering $f_1(z)=z^3, f_2(z)=z^4, g(z)=z^5$ and $c=-1$ in Theorem 2, we obtain the HQSF-surface

\begin{figure}[h!]
 \centering
       \begin{minipage}[b]{7.0cm}
        \includegraphics[width=5.5cm]{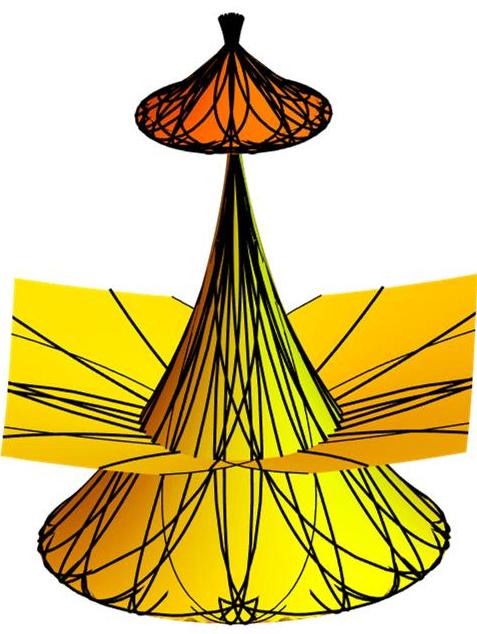}
       \caption{HQSF-surface}
      \end{minipage}
      \end{figure}}
 \end{example}
 \newpage
 \begin{example} {\em Considering $f_1(z)=\sin z, f_2(z)=\cos z, g(z)=e^{2z}$ and $c=-1$ in Theorem 2, we obtain the HQSF-surface

 \begin{figure}[h!]
 \centering
       \begin{minipage}[b]{7.0cm}
        \includegraphics[width=4.9cm]{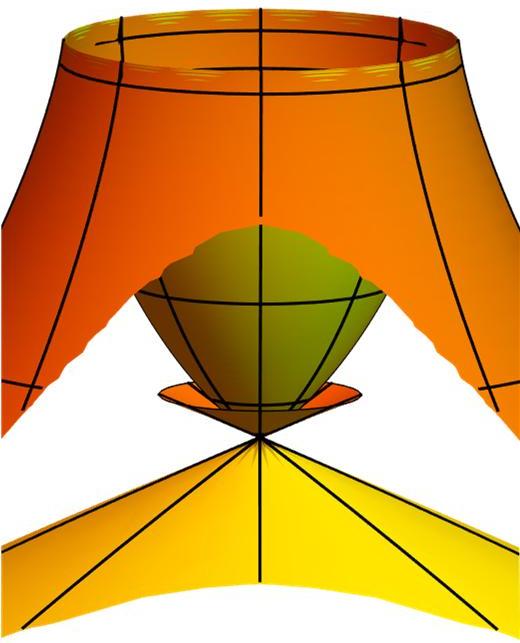}
       \caption{HQSF-surface}
      \end{minipage}
      \end{figure}}
 \end{example}
 \newpage
\begin{example} {\em Considering $f_1(z)=e^z, f_2(z)=-e^{-z}, g(z)=z$ and $c=-1$ in Theorem 2, we obtain the HQSF-surface

 \begin{figure}[h!]
 \centering
       \begin{minipage}[b]{7.0cm}
        \includegraphics[width=4.9cm]{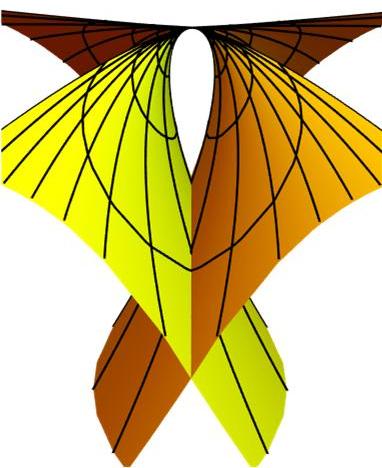}
       \caption{HQSF-surface}
      \end{minipage}
      \end{figure}}
 \end{example}

The following Theorem classify the HQSF-surfaces of rotation.\\

\begin{theorem} An oriented connected surface $S$ with nonzero Gauss curvature is a HQSF-surface of rotation if, and only if, locally can be parameterized by
\begin{equation}
\label{eqr1}
X_i(u)=(A_i(u_1)\cos u_2, A_i(u_1)\sin u_2,B_i(u_1)), i=1,2,3.
 \end{equation}
 where
\begin{eqnarray}
\label{eqr2}
A_i(u_1)&=&\frac{(1-e^{2u_1})h_i'(u_1)+2e^{2u_1}h_i(u_1)}{e^{u_1}(1+e^{2u_1})},\\
\label{eqr3}
B_i(u_1)&=&-\frac{2(h'_{i}(u_1)-h_i(u_1))}{1+e^{2u_1}},
\end{eqnarray}
\begin{equation}
\label{eqr121}
h_1(u_1)=\frac{e^{\frac{2\left(c c_1+c_2-\sqrt{\Omega}\right)u_1}{c}} \left(a_1^2 \left(\Omega+ (cc_1-c_2)\sqrt{\Omega}\right)+a_2^2 e^{\frac{2 \sqrt{\Omega} u_1}{c}} \left(\Omega-(c c_1-c_2)\sqrt{\Omega}\right)\right)^2}{4 c^2 \left(e^{2 u_1}+1\right)|z_1|^4},
\end{equation}
\begin{eqnarray}
\nonumber
h_2(u_1)&=&\frac{\Omega e^{2\left(\frac{c_2}c+c_1\right)u_1}}{16 c^2 \left(e^{2 u_1}+1\right)|z_1|^4} \left(\left(\left(b_1^2-b_2^2\right) (c c_1-c_2)+2 b_1 b_2 \sqrt{\Omega}\right)\sin \left(\frac{\sqrt{\Omega} u_1}{c}\right) \right.\\
\label{eqr141}
&& \left.+\left( \left(b_1^2-b_2^2\right)\sqrt{\Omega}-2 b_1 b_2 (c c_1-c_2)\right)\cos \left(\frac{\sqrt{\Omega} u_1}{c}\right) \right)^2,
\end{eqnarray}
\begin{equation}
\label{eqr161}
h_3(u_1)=\frac{e^{2\left(\frac{c_2}{c}+c_1\right)u_1} (a_2 ((a_1+a_2 u_1) (c_2-c c_1)+a_2 c))^2}{\left(e^{2 u_1}+1\right)|z_1|^4 }.
\end{equation}
\end{theorem}
\begin{proof}
Note that taking $g(w) = w, w \in \mathbb{C}$, $S$ is a HQSF-surface of rotation if, and only if, $h$ is a radial function i.e., $h(w) = r(|w|)$ for any differentiable function $r$.\\
Making the change of parameters 
\[
w = e^z,\, z = u_1 + iu_2 \in \mathbb{C},
\]
we have that $g(z)= e^z$ and $h_{,2} = 0$.\\
From (\ref{eq25}) and $h_{,2}=0$ we obtain 
\begin{equation}
\label{eqr4}
\langle f_1, i f_1'\rangle+ \langle f_2, i c f_2'\rangle=0. 
\end{equation}
We observed that when $c=0$, we obtain the QSF-surfaces of rotation studied in \cite{CRC}.\\
For $c\neq 0$, using the Lemma 1, we obtain
\begin{eqnarray}
\label{eqr5}
i f_1'&=&i c_1f_1-\overline{z_1}f_2,\\
\label{eqr6}
i c f_2'&=&z_1f_1+ic_2f_2. 
\end{eqnarray}
From (\ref{eqr5}) we obtain
\begin{equation}
\label{eqr7}
f_2=\frac{1}{\overline{z_1}}(i c_1f_1-f_1').
\end{equation}
Differentiating (\ref{eqr7})
\begin{equation}
\label{eqr8}
f_2'=\frac{1}{\overline{z_1}}(i c_1f_1'-f_1'').
\end{equation}
Substituting (\ref{eqr7}) and (\ref{eqr8}) in (\ref{eqr6})
\begin{equation}
\label{eqr9}
cf_1''-(c c_1+c_2)f_1'+(c_1c_2-|z_1|^2)f_1=0.
\end{equation}
The roots of the characteristic equation of equation (\ref{eqr9}) are given by
\begin{equation}
\label{eqr10}
\lambda=\frac{c c_1+c_2\pm \sqrt{\Omega}}{2c},
\end{equation}
where $\Omega= (cc_1-c_2)^2+4c|z_1|^2$.\\

{\bf CASE 1}: If $\Omega>0$ then the roots are given by
\[
\lambda_1=\frac{c c_1+c_2-\sqrt{\Omega}}{2c},\,\,\lambda_2=\frac{c c_1+c_2+\sqrt{\Omega}}{2c}.
\]
Thus
\begin{equation}
\label{eqr11}
f_1(z)=e^{\frac{(c c_1+c_2)}{2c}z}\left(a_1e^{-\frac{\sqrt{\Omega}}{2c}z}+a_2e^{\frac{\sqrt{\Omega}}{2c}z}\right).
\end{equation}
Using (\ref{eqr11}) in (\ref{eqr7})
\begin{equation}
\label{eqr12}
f_2(z)=\frac{i e^{\frac{(c c_1+c_2)}{2c}z}}{2c\overline{z_1}}\left(a_1(c c_1-c_2+\sqrt{\Omega})e^{-\frac{\sqrt{\Omega}}{2c}z}+a_2(cc_1-c_2-\sqrt{\Omega})e^{\frac{\sqrt{\Omega}}{2c}z}\right).
\end{equation}
From (\ref{eq25}), (\ref{eqr11}) and (\ref{eqr12}) we have (\ref{eqr121}).\\

{\bf CASE 2}: If $\Omega<0$ then the roots are given by
\[
\lambda_1=\frac{c c_1+c_2-i\sqrt{-\Omega}}{2c},\,\,\lambda_2=\frac{c c_1+c_2+i\sqrt{-\Omega}}{2c}.
\]
Thus
\begin{equation}
\label{eqr13}
f_1(z)=e^{\frac{(c c_1+c_2)}{2c}z}\left(b_1\cos{\left(\frac{\sqrt{-\Omega}}{2c}z\right)}+b_2\sin{\left(\frac{\sqrt{-\Omega}}{2c}z\right)}\right).
\end{equation}
Using (\ref{eqr13}) in (\ref{eqr7})
\begin{equation}
\label{eqr14}
f_2(z)=\frac{i e^{\frac{(c c_1+c_2)}{2c}z}}{2c\overline{z_1}}\left((b_1(c c_1-c_2)-b_2\sqrt{-\Omega})\cos\left({\frac{-\sqrt{\Omega}}{2c}z}\right)+(b_2(cc_1-c_2)+b_1\sqrt{-\Omega})\sin\left({\frac{\sqrt{-\Omega}}{2c}z}\right)\right).
\end{equation}
From (\ref{eq25}), (\ref{eqr13}) and (\ref{eqr14}) we have (\ref{eqr141}).\\

{\bf CASE 3}: If $\Omega=0$ then the roots are given by
\[
\lambda_1=\lambda_2=\frac{c c_1+c_2}{2c}.
\]
Thus
\begin{equation}
\label{eqr15}
f_1(z)=e^{\frac{(c c_1+c_2)}{2c}z}\left(a_1+a_2 z\right).
\end{equation}
Using (\ref{eqr15}) in (\ref{eqr7})
\begin{equation}
\label{eqr16}
f_2(z)=\frac{i e^{\frac{(c c_1+c_2)}{2c}z}}{2c\overline{z_1}}\left((cc_1-c_2)(a_1+a_2z)-2a_2c\right).
\end{equation}
From (\ref{eq25}), (\ref{eqr15}) and (\ref{eqr16}) we have (\ref{eqr161}).\\

From (\ref{eqr121}), (\ref{eqr141}), (\ref{eqr161}) and (\ref{eq14}), we obtain that $X$ is defined by (\ref{eqr1})-(\ref{eqr3}). The proof is complete

\end{proof}
\begin{example} {\em Considering $a_1=2, a_2=-\frac{1}{3}, c=1, c_1=0, c_2=1, z_1=-\frac{1}{2}$ in Theorem 3, we obtain 
\[
X_1(u_1)=(A_1(u_1)\cos u_2, A_1(u_1)\sin u_2,B_1(u_1)),
\]
and the profile curve is given by
\[
\alpha(u_1)=\left( A_1(u_1),0,B_1(u_1)\right),
\]
where
\begin{eqnarray*}
A_1(u_1)&=&\frac{8 e^{(1-2 \sqrt{2}) u_1} \left(\left(\sqrt{2}+2\right) e^{2 \sqrt{2} u_1}-36\sqrt{2}+72\right)}{81 \left(e^{2 u_1}+1\right)^3} \left(36 \sqrt{2} e^{4 u_1}-\sqrt{2} e^{2 \left(\sqrt{2}+2\right)u_1}\right.\\
&& \left.+\left(3 \sqrt{2}+4\right) e^{2 \sqrt{2} u_1}-108\sqrt{2}+144\right),\\
B_1(u_1)&=&\frac{8 e^{-2 \left(\sqrt{2}-1\right) u_1} \left(\left(\sqrt{2}+2\right) e^{2 \sqrt{2} u_1}-36\sqrt{2}+72\right)}{81 \left(e^{2 u_1}+1\right)^3} \left(36 \left(3 \sqrt{2}-2\right) e^{2 u_1}\right.\\
&&\left.-\left(3 \sqrt{2}+2\right) e^{2 \left(\sqrt{2}+1\right) u_1}-\left(5 \sqrt{2}+6\right) e^{2 \sqrt{2} u_1}-36 \left(6-5 \sqrt{2}\right)\right).
\end{eqnarray*}
The profile curve is regular, but it intersects the axis of rotation at two points, therefore, the HQSF-surface of rotation has two isolated singularities (see Figures 1 and 2).
\begin{figure}[h!]
 \centering
       \begin{minipage}[b]{6.0cm}
        \includegraphics[width=2.4cm]{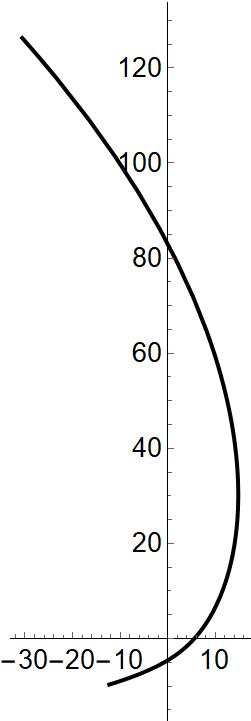}
       \caption{Profile curve}
      \end{minipage}
        \begin{minipage}[b]{6.5cm}
       \includegraphics[width=2.9cm]{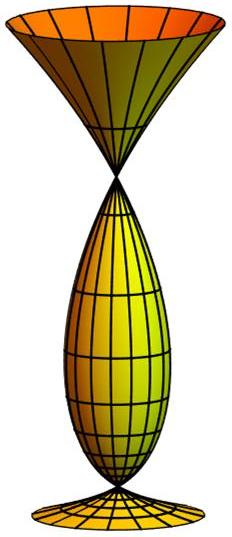}
      \caption{\smash{HQSF-surface of rotation}}
      \end{minipage}
      \end{figure}}
 \end{example}
 \begin{example} {\em Considering $a_1=1, a_2=-1, c=2, c_1=0, c_2=-\displaystyle\frac{1}{2}, z_1=\displaystyle\frac{i}{3}$ in Theorem 3, we obtain
  \[
X_1(u_1)=(A_1(u_1)\cos u_2, A_1(u_1)\sin u_2,B_1(u_1)),
\]
and the profile curve is given by
\[
\alpha(u_1)=\left( A_1(u_1),0,B_1(u_1)\right),
\]
where
\begin{eqnarray*}
A_1(u_1)&=&\frac{41 e^{-\frac{1}{6} \left(\sqrt{41}+9\right) u_1}}{384 \left(e^{2 u_1}+1\right)^3} \left(-48 e^{\frac{\sqrt{41} u_1}{6}}+432 e^{\frac{1}{6} \left(\sqrt{41}+24\right) u_1}+\left(17 \sqrt{41}-99\right) e^{\frac{\sqrt{41} u_1}{3}}\right.\\
&&\left.+\left(53 \sqrt{41}+399\right) e^{4 u_1}+\left(399-53 \sqrt{41}\right) e^{\frac{1}{3} \left(\sqrt{41}+12\right) u_1}-17 \sqrt{41}-99\right),\\
B_1(u_1)&=&\frac{41 e^{-\frac{u_1}{2}}}{48 \left(e^{2 u_1}+1\right)^3} \left(84 e^{2 u_1}-\sqrt{41} \left(22 e^{2 u_1}+13\right) \sinh \left(\frac{\sqrt{41} u_1}{6}\right)\right.\\
&&\left. +3 \left(54 e^{2 u_1}+29\right) \cosh \left(\frac{\sqrt{41} u_1}{6}\right)+36\right).
\end{eqnarray*}
The profile curve is regular, but it intersects the axis of rotation in one point, therefore, the HQSF-surface of rotation has an isolated singularity (see Figures 3 and 4).
\begin{figure}[h!]
 \centering
       \begin{minipage}[b]{6.0cm}
        \includegraphics[width=1.7cm]{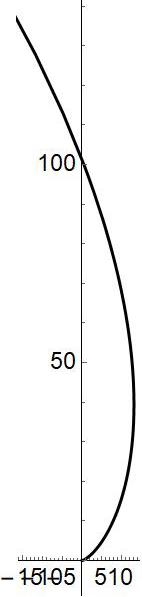}
       \caption{Profile curve}
      \end{minipage}
        \begin{minipage}[b]{6.5cm}
       \includegraphics[width=2.8cm]{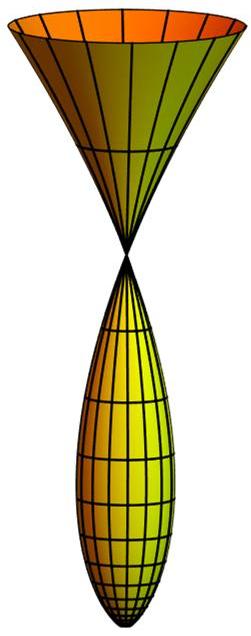}
      \caption{\smash{HQSF-surface of rotation}}
      \end{minipage}
      \end{figure}}
 \end{example}
 \begin{example} {\em Considering  $a_1=1, a_2=1, c=6, c_1=-1, c_2=\displaystyle\frac{1}{2}, z_1=3-\displaystyle\frac{i}{3}$ in Theorem 3, we obtain
  \[
X_1(u_1)=(A_1(u_1)\cos u_2, A_1(u_1)\sin u_2, B_1(u_1)),
\]
and the profile curve is given by
\[
\alpha(u_1)=\left( A_1(u_1),0,B_1(u_1)\right),
\]
where
\begin{eqnarray*}
A_1(u_1)&=&\frac{3131 e^{-\frac{1}{18} \left(\sqrt{9393}+51\right) u_1}}{7746048 \left(e^{2 u_1}+1\right)^3} \left(137760 e^{\frac{1}{18} \left(\sqrt{9393}+72\right) u_1}-43296 e^{\frac{1}{6} \sqrt{\frac{3131}{3}} u_1}\right.\\
&&\left.+\left(227 \sqrt{9393}+34443\right) e^{4 u_1}+\left(695 \sqrt{9393}+31041\right) e^{\frac{1}{3} \sqrt{\frac{3131}{3}} u_1}\right.\\
&&\left.+\left(34443-227 \sqrt{9393}\right) e^{\frac{1}{9} \left(\sqrt{9393}+36\right) u_1}-695 \sqrt{9393}+31041\right),\\
B_1(u_1)&=&\frac{3131 e^{-\frac{11 u_1}{6}}}{968256 \left(e^{2 u_1}+1\right)^3} \left(984 \left(29 e^{2 u-1}+17\right)-\sqrt{9393} \left(172 e^{2 u_1}+289\right)\times\right.\\
&&\left. \sinh \left(\frac{1}{6} \sqrt{\frac{3131}{3}} u_1\right)+9 \left(1004 e^{2 u_1}-815\right) \cosh \left(\frac{1}{6} \sqrt{\frac{3131}{3}} u_1\right)\right).
\end{eqnarray*}
The profile curve is not regular only in one point, therefore, the HQSF-surface of rotation has one isolated singularity and a circle of singularities (see Figures 5 and 6).
\begin{figure}[h!]
 \centering
       \begin{minipage}[b]{6.0cm}
        \includegraphics[width=1.8cm]{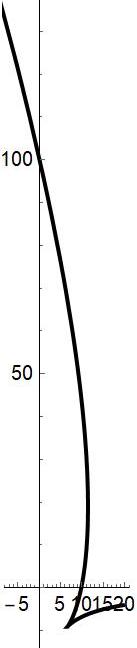}
       \caption{Profile curve}
      \end{minipage}
        \begin{minipage}[b]{6.5cm}
       \includegraphics[width=2.3cm]{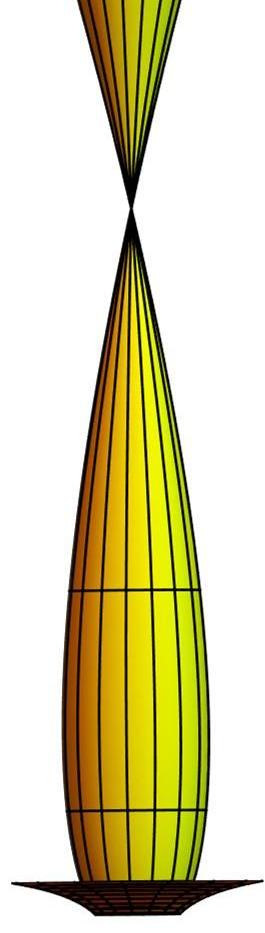}
      \caption{\smash{HQSF-surface of rotation}}
      \end{minipage}
      \end{figure}}
 \end{example}  
\begin{example} {\em Considering  $b_1=1, b_2=0, c=4, c_1=1, c_2=\displaystyle\frac{1}{2}, z_1=3$ in Theorem 3, we obtain
  \[
X_2(u_1)=(A_2(u_1)\cos u_2, A_2(u_1)\sin u_2, B_2(u_1)),
\]
and the profile curve is given by
\[
\alpha(u_1)=\left( A_2(u_1),0,B_2(u_1)\right),
\]
where
\begin{eqnarray*}
A_2(u_1)&=&\frac{625 e^{\frac{5 u_1}{4}} \left(7 \sin \left(\frac{25 u_1}{8}\right)+25 \cos \left(\frac{25 u_1}{8}\right)\right) \left(\left(337 e^{4 u_1}-281\right) \sin \left(\frac{25 u_1}{8}\right)+200 \cos \left(\frac{25 u_1}{8}\right)\right)}{663552 \left(e^{2 u_1}+1\right)^3},\\
B_2(u_1)&=&-\frac{625 e^{\frac{9 u_1}{4}}}{331776 \left(e^{2 u_1}+1\right)^3} \left(7 \sin \left(\frac{25 u_1}{8}\right)+25 \cos \left(\frac{25 u_1}{8}\right)\right) \left(50 \left(e^{2 u_1}+3\right) \cos \left(\frac{25 u_1}{8}\right)\right.\\
&&\left.-\left(323 e^{2 u_1}+295\right) \sin \left(\frac{25 u_1}{8}\right)\right).
\end{eqnarray*}
 The profile curve is not regular in nine points, therefore, the HQSF-surface of rotation has nine circles of singularities and three isolated singularities (see Figures 7 and 8).
\begin{figure}[h!]
 \centering
       \begin{minipage}[b]{6.0cm}
        \includegraphics[width=4.2cm]{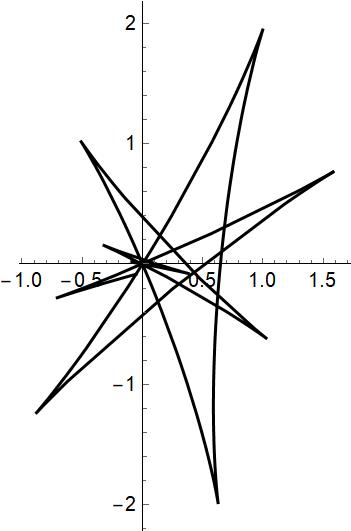}
       \caption{Profile curve}
      \end{minipage}
        \begin{minipage}[b]{6.5cm}
       \includegraphics[width=4.3cm]{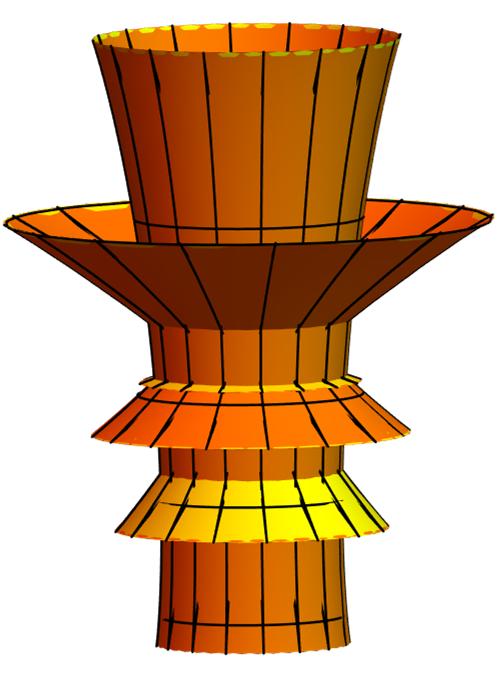}
      \caption{\smash{HQSF-surface of rotation}}
      \end{minipage}
      \end{figure}}
 \end{example}      
 \begin{example} {\em Considering  $b_1=0, b_2=2, c=4, c_1=-1, c_2=-2, z_1=-1$ in Theorem 3, we obtain
  \[
X_2(u_1)=(A_2(u_1)\cos u_2, A_2(u_1)\sin u_2, B_2(u_1)),
\]
and the profile curve is given by
\[
\alpha(u_1)=\left( A_2(u_1),0,B_2(u_1)\right),
\]
where
\begin{eqnarray*}
A_2(u_1)&=&\frac{e^{-u_1} \left(96 e^{4 u_1}-16 \sqrt{2} \left(5 e^{4 u_1}-1\right) \sin \left(2 \sqrt{2} u_1\right)+32 \left(e^{4 u_1}-2\right) \cos \left(2 \sqrt{2} u_1\right)\right)}{\left(e^{2 u_1}+1\right)^3},\\
B_2(u_1)&=&\frac{16 \left(9 e^{2 u_1}-4 \sqrt{2} \left(2 e^{2 u_1}+1\right) \sin \left(2 \sqrt{2} u_1\right)+\left(5 e^{2 u_1}+7\right) \cos \left(2 \sqrt{2} u_1\right)+3\right)}{\left(e^{2 u_1}+1\right)^3}.
\end{eqnarray*}
 The profile curve is not regular only in two points, therefore, the HQSF-surface of rotation has two circles of singularities and three isolated singularities (see Figures 9 and 10).
\begin{figure}[h!]
 \centering
       \begin{minipage}[b]{6.0cm}
        \includegraphics[width=2.4cm]{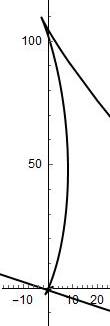}
       \caption{Profile curve}
      \end{minipage}
        \begin{minipage}[b]{6.5cm}
       \includegraphics[width=2.5cm]{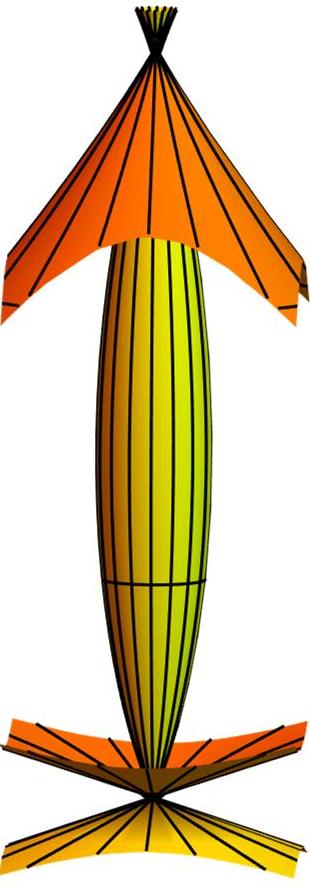}
      \caption{\smash{HQSF-surface of rotation}}
      \end{minipage}
      \end{figure}}
 \end{example}        
\begin{example} {\em Considering  $b_1=1, b_2=-1, c=3, c_1=1, c_2=-1, z_1=1-i$ in Theorem 3, we obtain
  \[
X_2(u_1)=(A_2(u_1)\cos u_2, A_2(u_1)\sin u_2, B_2(u_1)),
\]
and the profile curve is given by
\[
\alpha(u_1)=\left( A_2(u_1),0,B_2(u_1)\right),
\]
where
\begin{eqnarray*}
A_2(u_1)&=&\frac{40 e^{\frac{u_1}{3}} \left(14 e^{4 u_1}-\sqrt{10} \left(7 e^{4 u_1}-1\right) \sin \left(\frac{4 \sqrt{10} u_1}{3}\right)+\left(14 e^{4 u_1}-23\right) \cos \left(\frac{4 \sqrt{10} u_1}{3}\right)+7\right)}{27 \left(e^{2 u_1}+1\right)^3},\\
B_2(u_1)&=&\frac{20 e^{\frac{4 u_1}{3}} \left(35 e^{2 u_1}-2 \sqrt{10} \left(11 e^{2 u_1}+5\right) \sin \left(\frac{4 \sqrt{10} u_1}{3}\right)+\left(65 e^{2 u_1}+83\right) \cos \left(\frac{4 \sqrt{10} u_1}{3}\right)-7\right)}{27 \left(e^{2 u_1}+1\right)^3}.
\end{eqnarray*}
 The profile curve is not regular in six points, therefore, the HQSF-surface of rotation has six circles of singularities and three isolated singularities (see Figures 11 and 12).
\begin{figure}[h!]
 \centering
       \begin{minipage}[b]{6.0cm}
        \includegraphics[width=5.7cm]{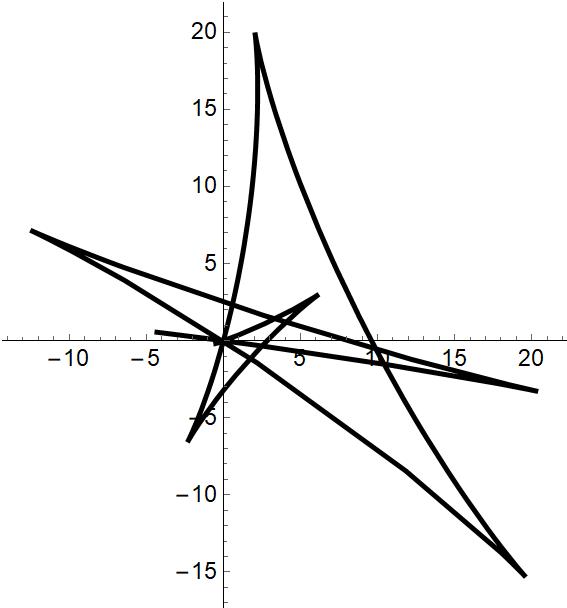}
       \caption{Profile curve}
      \end{minipage}
        \begin{minipage}[b]{6.5cm}
       \includegraphics[width=5.9cm]{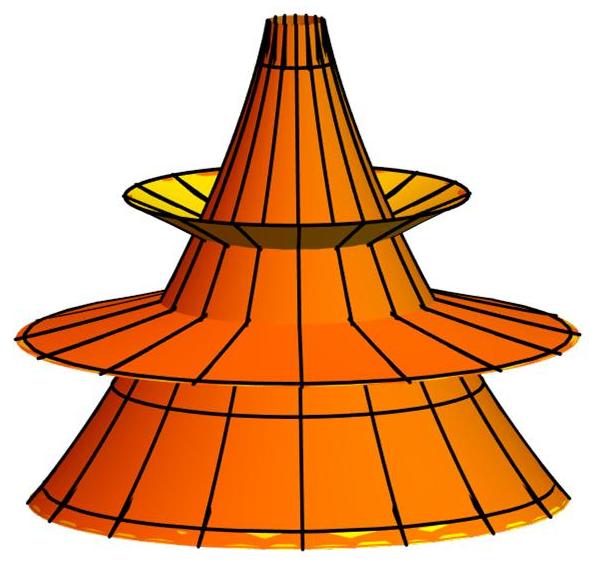}
      \caption{\smash{HQSF-surface of rotation}}
      \end{minipage}
      \end{figure}}
 \end{example}      
\begin{example} {\em Considering  $a_1=1, a_2=1, c=2, c_1=1, c_2=2, z_1=3$ in Theorem 3, we obtain
  \[
X_3(u_1)=(A_3(u_1)\cos u_2, A_3(u_1)\sin u_2, B_3(u_1)),
\]
and the profile curve is given by
\[
\alpha(u_1)=\left( A_3(u_1),0,B_3(u_1)\right),
\]
where
\begin{eqnarray*}
A_3(u_1)&=&\frac{16 e^{3 u_1}}{81 \left(e^{2 u_1}+1\right)^3},\\
B_3(u_1)&=&-\frac{8 e^{4 u_1} \left(e^{2 u_1}+3\right)}{81 \left(e^{2 u_1}+1\right)^3}.
\end{eqnarray*}
 The profile curve is regular, therefore, the HQSF-surface of rotation is complete (see Figures 13 and 14).
\begin{figure}[h!]
 \centering
       \begin{minipage}[b]{6.0cm}
        \includegraphics[width=2.5cm]{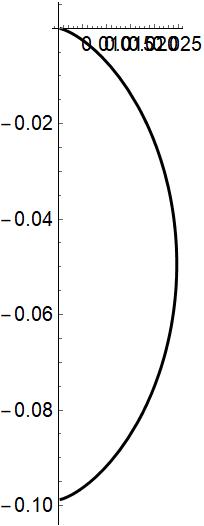}
       \caption{Profile curve}
      \end{minipage}
        \begin{minipage}[b]{6.5cm}
       \includegraphics[width=2.97cm]{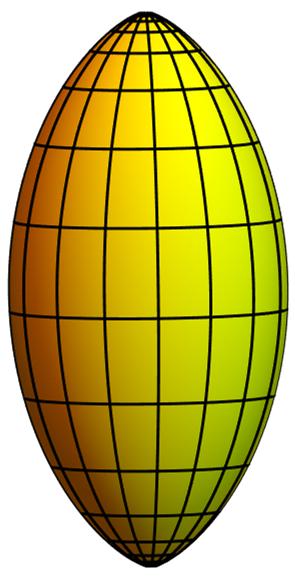}
      \caption{\smash{HQSF-surface of rotation}}
      \end{minipage}
      \end{figure}}
 \end{example}      
 
 \begin{example} {\em Considering  $a_1=1, a_2=1, c=1, c_1=-1, c_2=\frac{1}{2}, z_1=3-\frac{i}{3}$ in Theorem 3, we obtain
  \[
X_3(u_1)=(A_3(u_1)\cos u_2, A_3(u_1)\sin u_2, B_3(u_1)),
\]
and the profile curve is given by
\[
\alpha(u_1)=\left( A_3(u_1),0,B_3(u_1)\right),
\]
where
\begin{eqnarray*}
A_3(u_1)&=&\frac{81 e^{-2 u_1} (3 u_1+5) \left(-3 u_1+e^{4 u_1} (15 u_1+19)+1\right)}{26896 \left(e^{2 u_1}+1\right)^3},\\
B_3(u_1)&=&\frac{81 (3 u_1+5) ((3 u_1+5) \sinh (u_1)+9 (u_1+1) \cosh (u_1))}{6724 \left(e^{2 u_1}+1\right)^3}.
\end{eqnarray*}
 The profile curve is not regular only in two points, therefore, the HQSF-surface of rotation has one isolated singularity and two circles of singularities (see Figures 15 and 16).
\begin{figure}[h!]
 \centering
       \begin{minipage}[b]{6.5cm}
        \includegraphics[width=6.5cm]{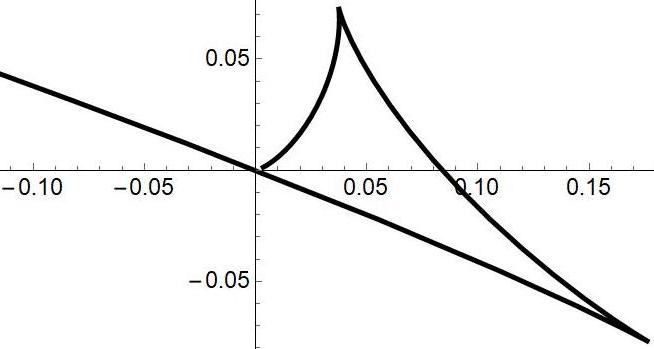}
       \caption{Profile curve}
      \end{minipage}
        \begin{minipage}[b]{6.0cm}
       \includegraphics[width=5.97cm]{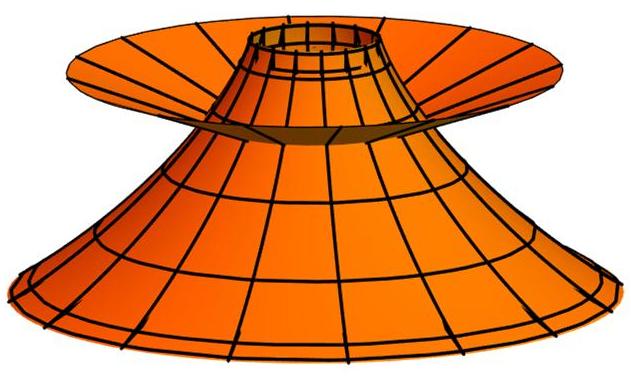}
      \caption{\smash{HQSF-surface of rotation}}
      \end{minipage}
      \end{figure}}
 \end{example}      
 \begin{example} {\em Considering  $a_1=\frac{1}{2}, a_2=-\frac{1}{3}, c=-1, c_1=0, c_2=-\frac{1}{2}, z_1=-1+\frac{i}{3}$ in Theorem 3, we obtain
  \[
X_3(u_1)=(A_3(u_1)\cos u_2, A_3(u_1)\sin u_2, B_3(u_1)),
\]
and the profile curve is given by
\[
\alpha(u_1)=\left( A_3(u_1),0,B_3(u_1)\right),
\]
where
\begin{eqnarray*}
A_3(u_1)&=&\frac{(2 u_1+1) \left(2 u_1+e^{4 u_1} (6 u_1-1)+5\right)}{1600 \left(e^{2 u_1}+1\right)^3},\\
B_3(u_1)&=&\frac{e^{u_1} (2 u_1+1) \left(e^{2 u_1} (2 u_1-1)-2\right)}{400 \left(e^{2 u_1}+1\right)^3}.
\end{eqnarray*}
 The profile curve is not regular only in two points, therefore, the HQSF-surface of rotation has one isolated singularity and two circles of singularities (see Figures 5 and 6).
\begin{figure}[h!]
 \centering
       \begin{minipage}[b]{6.0cm}
        \includegraphics[width=5.7cm]{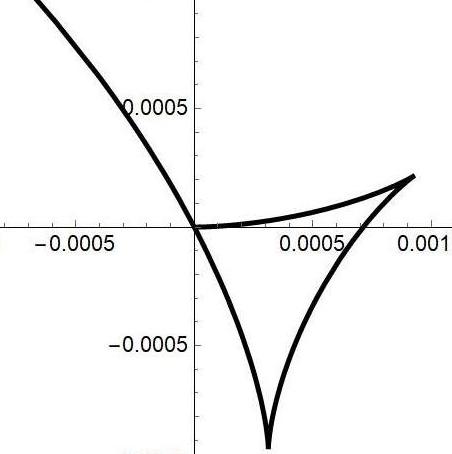}
       \caption{Profile curve}
      \end{minipage}
        \begin{minipage}[b]{6.0cm}
       \includegraphics[width=5.1cm]{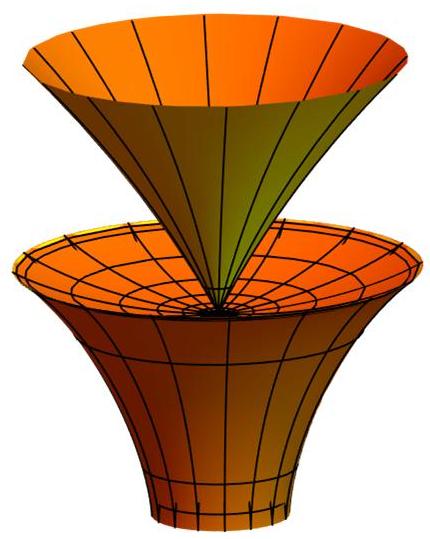}
      \caption{\smash{HQSF-surface of rotation}}
      \end{minipage}
      \end{figure}}
 \end{example}

\end{document}